\documentclass[reqno]{amsart}

\input xy
\xyoption{all}
\usepackage{epsfig}
\usepackage{color}
\usepackage{amsthm}
\usepackage{amssymb}
\usepackage{amsmath}
\usepackage{amscd}
\usepackage{amsopn}
\usepackage{url}
\usepackage{hyperref}\hypersetup{colorlinks}

\usepackage{stmaryrd}

\usepackage[T1]{fontenc}

\usepackage{color} 

\definecolor{darkred}{rgb}{1,0,0} 
\definecolor{darkgreen}{rgb}{0,.8,0}
\definecolor{darkblue}{rgb}{0,0,1}

\hypersetup{colorlinks,
linkcolor=darkblue,
filecolor=darkgreen,
urlcolor=darkred,
citecolor=darkgreen}

\numberwithin{equation}{section}
\newtheorem {Theorem}{Theorem}
\numberwithin{Theorem}{section}

\newtheorem {Corollary}[Theorem]{Corollary}
\theoremstyle{definition}

\theoremstyle{remark}
\newtheorem{Remark}[Theorem]{Remark}

\expandafter\chardef\csname pre amssym.def at\endcsname=\the\catcode`\@
\catcode`\@=11
\def\undefine#1{\let#1\undefined}
\def\newsymbol#1#2#3#4#5{\let\next@\relax
 \ifnum#2=\@ne\let\next@\msafam@\else
 \ifnum#2=\tw@\let\next@\msbfam@\fi\fi
 \mathchardef#1="#3\next@#4#5}
\def\mathhexbox@#1#2#3{\relax
 \ifmmode\mathpalette{}{\m@th\mathchar"#1#2#3}%
 \else\leavevmode\hbox{$\m@th\mathchar"#1#2#3$}\fi}
\def\hexnumber@#1{\ifcase#1 0\or 1\or 2\or 3\or 4\or 5\or 6\or 7\or 8\or
 9\or A\or B\or C\or D\or E\or F\fi}

\font\teneufm=eufm10
\font\seveneufm=eufm7
\font\fiveeufm=eufm5
\newfam\eufmfam
\textfont\eufmfam=\teneufm
\scriptfont\eufmfam=\seveneufm
\scriptscriptfont\eufmfam=\fiveeufm

\catcode`\@=\csname pre amssym.def at\endcsname

\def    \eps    {\epsilon}

\newcommand{\CA}{{\mathcal A}}

\newcommand{\CS}{{\mathcal S}}

\newcommand{\const}{{\mathit const}}

\newcommand{\Pp}{{\mathcal P}}

\def    \nat    {{\natural}}

\def    \R      {{\mathbb R}}

\def    \Z      {{\mathbb Z}}
\def    \N      {{\mathbb N}}

\def    \T      {{\mathbb T}}
\def    \CP     {{\mathbb C}{\mathbb P}}

\def    \12    {{\frac{1}{2}}}

\def    \HF     {\operatorname{HF}}

\def    \H      {\operatorname{H}}
\def    \Tor    {\operatorname{Tor}}

\def  \ka {\kappa}

\def  \defn  {\mathrel{\mathop:}=}

\def  \MUCZ  {\operatorname{\mu_{\scriptscriptstyle{CZ}}}}

\begin{document}


\setlength{\smallskipamount}{6pt}
\setlength{\medskipamount}{10pt}
\setlength{\bigskipamount}{16pt}





\title[Non-contractible Periodic Orbits]{On Non-contractible Periodic
  Orbits of Hamiltonian Diffeomorphisms}

\author[Ba\c sak G\"urel]{Ba\c sak Z. G\"urel}

\address{BG: Department of Mathematics, Vanderbilt University,
  Nashville, TN 37240, USA} 
\email{basak.gurel@vanderbilt.edu}


\date{\today} 

\thanks{The work is partially supported by the NSF grants DMS-0906204
  and DMS-1207680.}


\begin{abstract} 
  We prove that any Hamiltonian diffeomorphism of a closed symplectic
  manifold equipped with an atoroidal symplectic form has simple
  non-contractible periodic orbits of arbitrarily large period,
  provided that the diffeomorphism has a non-degenerate (or even
  isolated and homologically non-trivial) periodic orbit with non-zero
  homology class and the set of one-periodic orbits in that class is
  finite.
\end{abstract}

\maketitle

\tableofcontents
\section{Introduction and main results}
\label{sec:intro}

In this paper, we study the question of existence of non-contractible
periodic orbits for Hamiltonian diffeomorphisms $\varphi_H$ of a
closed symplectic manifold $(M^{2n}, \omega)$ such that the integral
of $\omega$ over all toroidal classes in $\H_2(M;\Z)$ vanishes. Under
this condition, we prove that $\varphi_H$ has non-contractible
periodic orbits of arbitrarily large prime period, provided that it
has one non-degenerate (or isolated and homologically non-trivial)
periodic orbit with homology class non-zero modulo torsion.

Non-contractible periodic orbits of Hamiltonian systems have been
previously investigated in a number of papers using Floer theoretical
techniques; see, e.g., \cite{BPS,BuHa,GL,Lee,Ni,We}.  The central
theme of this paper, however, lies in a different domain.
Conceptually, the key difference is that here we allow the period of
the orbits to vary, whereas the above works concern the existence of
periodic orbits (in a given free homotopy class) of a fixed period.
In other words, we approach the problem from a dynamical systems
perspective rather than topological.  On the technical side, the
source of periodic orbits in those works is non-vanishing Floer
homology or other more subtle invariants such as Floer-theoretic
torsion.  In general, one major obstacle to employing this method in
the investigation of non-contractible periodic orbits for closed
manifolds is that the total non-contractible Floer homology is zero
since all one-periodic orbits of a $C^2$-small autonomous Hamiltonian
are constant, and hence contractible. (Thus even the entire Floer
complex may be zero.)  Therefore, the previous results mainly apply to
open manifolds such as ordinary and twisted cotangent bundles; see,
e.g., \cite{BPS,Ni,We}.  Working with closed manifolds, we take a
different approach. In our case, the source for periodic orbits is the
change in certain filtered Floer homology groups under the iteration
of the Hamiltonian, and the original non-contractible orbit spawns
infinitely many other orbits. (Our argument shares many common
elements with that \cite{Gu:hq}.)

To put our result in perspective in a different way, recall the
following variant of the Conley conjecture, referred to as the
HZ-conjecture in \cite{Gu:hq}. This is the assertion that a
Hamiltonian diffeomorphism with ``more than necessary'' non-degenerate
(or just homologically non-trivial) fixed points has infinitely many
periodic orbits. Here ``more than necessary'' is usually interpreted
as a lower bound arising from some version of the Arnold
conjecture. For $\CP^n$, the expected threshold is $n+1$.  This
conjecture, originally stated in \cite{HZ}, is inspired by a
celebrated theorem of Franks stating that a Hamiltonian diffeomorphism
(or, even, an area preserving homeomorphism) of $S^2$ with at least
three fixed points must have infinitely many periodic orbits,
\cite{Fr1,Fr2}; see also \cite{FrHa,LeC} for further refinements and
\cite{BH,CKRTZ,Ke:JMD} for symplectic topological proofs. 

Viewing the HZ-conjecture in a broader context is often
useful. Namely, it appears that the presence of a fixed point that is
unnecessary from a homological or geometrical perspective is already
sufficient to force the existence of infinitely many periodic orbits.
For instance, a theorem from \cite{GG:hyperbolic} asserts that, for a
certain class of closed monotone symplectic manifolds including
$\CP^n$, any Hamiltonian diffeomorphism with a hyperbolic fixed point
must necessarily have infinitely many periodic orbits; see also
\cite{Gu:hq} for other results along these lines. (Note also that
these considerations are related to and motivated by the circle of
questions in the realm of the original Conley conjecture; see,
\cite{FrHa,Gi:conley,GG:neg-mon, Hi,LeC,SZ} and also
\cite{CGG,GG:gaps,He:irr} for a detailed account of this conjecture
and related results.)

The main result of the paper therefore can be viewed as the proof of
the generalized HZ-conjecture for non-contractible periodic orbits of
Hamiltonian diffeomorphisms. Indeed, observe that from a homological
perspective a Hamiltonian diffeomorphism of a closed symplectic
manifold need not have a non-contractible periodic orbit. However, our
result shows that, if such an orbit exists, it generates infinitely
many others.

Let us now state the main result of the paper. Let $(M^{2n}, \omega)$
be a closed symplectic manifold equipped with an \emph{atoroidal}
symplectic form $\omega$. (Namely, we assume that for every map
$v\colon \T^2\to M$, the integral of $\omega$ over $v$ vanishes, i.e.,
$\left< [\omega], [v] \right>=0$.)  For a Hamiltonian diffeomorphism
$\varphi_H$ of $M$, denote by $\Pp_1(\varphi_H,[\gamma])$ the
collection of one-periodic orbits of $\varphi_H$ with homology class
$[\gamma] \in \H_1(M;\Z)/\Tor$.

\begin{Theorem}
\label{thm:main}
Let $M$ be a closed symplectic manifold equipped with an atoroidal
symplectic form $\omega$. Assume that a Hamiltonian diffeomorphism
$\varphi_H$ of $M$ has a non-degenerate one-periodic orbit $\gamma$
with homology class $[\gamma] \neq 0$ in $\H_1(M;\Z)/\Tor$ and that
$\Pp_1(\varphi_H,[\gamma])$ is finite.  Then, for every sufficiently
large prime $p$, the diffeomorphism $\varphi_H$ has a simple periodic
orbit with homology class $p[\gamma]$ and period either $p$ or $p'$,
where $p'$ is the first prime number greater than $p$.
\end{Theorem}

\begin{Corollary} 
\label{cor:growth}
In the setting of Theorem \ref{thm:main} (or Theorem \ref{thm:main2}),
the number of non-contractible periodic orbits with period less than
or equal to $k$, or the number of distinct homology classes
represented by such orbits, is bounded from below by $ \const \cdot
k/\log k$.
\end{Corollary}

An immediate consequence of Theorem \ref{thm:main} is that $\varphi_H$
has infinitely many periodic orbits with homology classes in
$\N[\gamma]$ whether or not $\Pp_1(\varphi_H,[\gamma])$ is finite.
Moreover, the non-degeneracy condition in the theorem can be relaxed
and replaced by a much weaker, albeit more technical, requirement that
$\gamma$ is isolated and \emph{homologically non-trivial}, i.e., its
local Floer homology is non-zero. (See Section \ref{sec:proofs} for a
discussion of this notion.) Finally, the orbit $\gamma$ need not be
one-periodic; the theorem (with obvious modifications) still holds
when $\gamma$ is just a periodic orbit.

Among the manifolds meeting the requirements of the theorem are, for
instance, all closed K\"ahler manifolds with negative sectional
curvature (e.g., complex hyperbolic spaces) or, more generally, any
closed symplectic manifold $M$ such that $\pi_2(M)=0$ and $\pi_1(M)$
is a hyperbolic group. Indeed, in this case any map $v \colon \T^2 \to
M$ is homologous to zero over $\R$, since $\pi_1(M)$ contains no
subgroups isomorphic to $\Z \oplus \Z$ (see, e.g.,
\cite{BriHae,GhdlH}), and hence the pull-back homomorphism
$$
v^* \colon\H^2(M;\R)=\H^2(K(\pi_1(M,1);\R))\to \H^2(\T^2;\R)
$$ 
is necessarily zero. (See \cite{BruK,Ked} and also references therein
for other examples.)

The existence of a non-contractible periodic orbit $\gamma$ as in
Theorem \ref{thm:main} appears to be a common occurrence. In fact, we
are not aware of any example of a Hamiltonian diffeomorphism with
isolated fixed points which obviously would not have a
non-contractible periodic orbit, provided that $\H_1(M;\R)\neq
0$. However, let us emphasize that the homology class of $\gamma$
cannot be prescribed in advance, for, as simple examples (e.g., the
height function on $\T^2$) show, non-contractible orbits need not
exist in a given homology class or in the multiples of a given
class. (In fact, KAM theory implies that this is even generically the
case, at least in dimension two.)

The growth rate established in Corollary \ref{cor:growth} is typical
in this class of questions in the absence of hyperbolicity. For
instance, a similar growth rate is known to hold in the context of the
Conley conjecture (see, e.g., \cite{Gi:conley,GG:gaps,He:irr,Hi,SZ})
and for closed geodesics on $S^2$; see, e.g., \cite{Hi:geo}. (There
are, however, stronger growth results; see, e.g., \cite{LeC,Vi:gen}.)
Perhaps it is also worth pointing out that Theorem \ref{thm:main} and
Corollary \ref{cor:growth} are reminiscent of some results on the
growth of prime geodesics on negatively curved manifolds (with the
scale exponentially adjusted); see, e.g., \cite{AS,Ma,PS}. However, to
the best our understanding, there is essentially no connection,
technical or conceptual, between the respective results beyond a
superficial resemblance.

Finally, the question of possible generalizations of Theorem
\ref{thm:main} to the case when $M$ is toroidally monotone will be
addressed elsewhere.  (Recall that $M$ is said to be toroidally
monotone if for every map $v\colon \T^2\to M$ we have $\left<[\omega],
  [v]\right>=\lambda\left<c_1(TM),[v]\right>$ for some constant
$\lambda\geq 0$ independent of $v$. See \cite{GG:hyperbolic} for a
partial result in this direction.)

\subsection{Acknowledgements} The author is grateful to Viktor
Ginzburg, Denis Osin, and Ioana \c{S}uvaina for useful discussions.

\section{Preliminaries}
\label{sec:prelim}
In this section, we set conventions and notation, and discuss Floer
homology for non-contractible periodic orbits.

\subsection{Conventions and notation}
\label{sec:conventions}
Throughout the paper, we assume that $(M^{2n}, \omega)$ is a closed
manifold equipped with an \emph{atoroidal} symplectic form $\omega$,
i.e., for every map $v\colon \T^2\to M$, the integral of $\omega$ over
$v$ vanishes: $\left< [\omega], [v] \right>=0$.  We refer the reader
to Section \ref{sec:intro} for examples of such (open or closed)
symplectic manifolds and here only note that they arise naturally as
the simplest setting where Floer theory for non-contractible periodic
orbits is defined.

All Hamiltonians $H$ are assumed to be one-periodic in time, i.e.,
$H\colon S^1\times M\to\R$, and we set $H_t = H(t,\cdot)$ for $t\in
S^1=\R/\Z$.  The Hamiltonian vector field $X_H$ of $H$ is defined by
$i_{X_H}\omega=-dH$. The (time-dependent) flow of $X_H$ is denoted by
$\varphi_H^t$ and its time-one map by $\varphi_H$. Such time-one maps
are referred to as \emph{Hamiltonian diffeomorphisms}.  The Hofer norm
of $H$ is $ \| H \| = \int_{S^1} \left[ \max_M H_t - \min_M H_t
\right] \, dt.  $

Let $K$ and $H$ be two one-periodic Hamiltonians. The composition
$K\nat H$ is defined by the formula
$$
(K\nat H)_t=K_t+H_t\circ(\varphi^t_K)^{-1},
$$
and the flow of $K\nat H$ is $\varphi^t_K\circ \varphi^t_H$. We set
$H^{\nat k}=H\nat\ldots \nat H$ ($k$ times).  Abusing terminology, we
will refer to $H^{\nat k}$ as the $k$th iteration of $H$.  (Note that
the flow $\varphi^t_{H^{\nat k}}=(\varphi_H^t)^k$, $t\in [0,\,1]$, is
homotopic with fixed end-points to the flow $\varphi^t_H$, $t\in [0,\,
k]$.)

In general, $H^{\nat k}$ is not one-periodic, even when $H$
is. However, this is the case if, for example, $H_0\equiv 0\equiv
H_1$. The latter condition can be met by reparametrizing the
Hamiltonian as a function of time without changing the time-one map;
the Hofer norm, and actions and indices of the periodic orbits do not
change during the procedure.  Thus, in what follows, we always treat
$H^{\nat k}$ as a one-periodic Hamiltonian.

The $k$th iteration of a one-periodic orbit $\gamma$ of $H$ is denoted
by $\gamma^k$. More specifically, $\gamma^k(t)=\varphi_{H^{\nat k}}^t
\left(\gamma(0)\right)$, where $t\in [0,\,1]$.  We can think of
$\gamma^k$ as the $k$-periodic orbit $\gamma(t)$, $t\in [0,\,k]$, of
$H$.  Hence, there is an action--preserving one-to-one correspondence
between one-periodic orbits of $H^{\nat k}$ and $k$-periodic orbits of
$H$.

\subsection{Floer homology for non-contractible periodic orbits}
\label{sec:Floer}
The key tool used in the proof of Theorem \ref{thm:main} is the
filtered Floer homology for non-contractible periodic orbits of
Hamiltonian diffeomorphisms.  Various flavors of Floer homology in
this case for both open and closed manifolds have been considered in
several other works; see, e.g.,
\cite{BPS,BuHa,GL,GG:hyperbolic,Lee,Ni,We}.  Below we briefly describe
the elements of the construction relevant to our case.

Let $\zeta$ be a free homotopy class of maps $S^1\to M$. Fix a
\emph{reference loop} $z\in\zeta$ and a symplectic trivialization of
$TM$ along $z$. (In fact, it would be sufficient to fix a
trivialization of the canonical bundle of $M$ along $z$.)  A capping
of $x \colon S^1 \to M$ with free homotopy class $\llbracket x
\rrbracket = \zeta$ is a cylinder (i.e., a homotopy) $\Pi \colon [0,1]
\times S^1 \to M$ connecting $x$ and $z$.  The action of a
one-periodic Hamiltonian $H$ on $x$ is 
$$
\CA_H(x)=
-\int_\Pi\omega+\int_{S^1} H_t(x(t))\,dt.
$$
The action $\CA_H(x)$ is well-defined, i.e., independent of capping,
for $\omega$ is atoroidal and hence $\left<[\omega],[v]\right>=0$,
where $v\colon \T^2\to M$ is the torus obtained by attaching two
different cappings to each other. Moreover, the critical points of
$\CA_H$ are exactly one-periodic orbits of the time-dependent
Hamiltonian flow $\varphi_H^t$ with homotopy class in $\zeta$.  The
action spectrum $\CS(H,\zeta)$ is the set of critical values of
$\CA_H$. It has zero measure; see, e.g., \cite{HZ}.

Observe that the trivialization of $TM$ along the reference loop $z$
uniquely extends to a trivialization along $\Pi$, and hence induces a
trivialization along $x$. Using this trivialization, the
\emph{Conley--Zehnder index} $\MUCZ(H,x)\in\Z$ of a
\emph{non-degenerate} orbit $x$ is defined as in \cite{Sa,SZ}. (Here
$x$ is said to be non-degenerate if the linearized return map
$d\varphi_H \colon T_{x (0)}M\to T_{x (0)}M$ does not have one as an
eigenvalue.) Similarly to the contractible case, the Conley-Zehnder
index is defined only modulo $2N$, where $N$ is the minimal
``toroidal'' Chern number of $M$, unless $c_1(TM)$ vanishes on
toroidal homology classes in $\H_2(M;\Z)$. Note that, by definition,
$N$ is the positive generator of the subgroup of $\Z$ generated by the
integrals of $c_1(TM)$ over all toroidal classes.

The construction of Floer homology $\HF(H,\zeta)$ and the filtered
Floer homology $\HF^{(a,\, b)}(H,\zeta)$, taking into account only the
one-periodic orbits with homotopy classes in $\zeta$, goes through
exactly as in the contractible case; see, e.g., \cite{BPS,Lee}.

Recall that the total Floer homology in a class $\zeta \neq 0$ is
trivial: $\HF(H,\zeta)=0$; for all one-periodic orbits of a
$C^2$-small autonomous Hamiltonian are constant, and hence
contractible.  However, the filtered Floer homology for some action
interval is obviously non-zero whenever there exists an isolated and
homologically non-trivial (e.g., non-degenerate) periodic orbit
$\gamma \in \zeta$.  Indeed, assuming for simplicity that
$\CA_H(\gamma)=0$, observe that
$\HF^{(-a,\,a)}(H,\llbracket\gamma\rrbracket)\neq0$, for instance,
when $(-a,\,a)$ contains only the action value zero for $a \not\in
\CS(H,\llbracket\gamma\rrbracket)$. This follows from the fact that
\begin{equation*}
\HF^{(-a,\,a)}(H,\llbracket\gamma\rrbracket)=
\HF(\gamma) \oplus \hdots \neq 0,
\end{equation*}
where $\HF(\gamma)$ is the local Floer homology of $H$ at $\gamma$
(see, e.g., \cite{F:89witten,Fl89fp,Gi:conley,GG:gap}), and the dots
represent the local Floer homology contributions from other orbits in
the class $\llbracket\gamma\rrbracket$ with action equal to zero; see,
e.g., \cite{GG:gap} for a proof of this fact.

The action and the index of periodic orbits (and hence the grading and
filtration of the Floer complex) do depend on the choice of reference
curves $z$ and, in the index case, on the trivialization of $TM$ along
$z$.  Whenever we consider the iteration $H^{\nat \ka}$ of $H$, we
simultaneously iterate each class $\zeta$ and the reference curve $z$
(i.e., we pass to $\zeta^ \ka$ and $z^\ka$). As a consequence, the
action functional on \emph{iterated} $\ka$-periodic orbits is
homogeneous with respect to the iterations, i.e.,
\begin{equation}
\label{eq:action-hom}
\CA_{H^{\nat \ka}}(x^{\ka}) = \ka \CA_H(x).
\end{equation}

\section{Proof of Theorem \ref{thm:main}}
\label{sec:proofs}
As has been mentioned in the introduction, we establish a more general
result. To state it, recall that an isolated periodic orbit $\gamma$
of $\varphi_H$ is said to be homologically non-trivial if the local
Floer homology of $H$ at $\gamma$ is non-zero. For instance, a
non-degenerate fixed point is homologically non-trivial. More
generally, an isolated fixed point with non-vanishing topological
index is homologically non-trivial; for this index is equal, up to a
sign, to the Euler characteristic of the local Floer homology.  The
notion of homological non-triviality seems to be particularly
well-suited for use in the context of Conley conjecture-type
questions; we refer the reader to \cite{Gu:hq} for a detailed
discussion of this condition.  Theorem \ref{thm:main} is an immediate
consequence of the following result.
\begin{Theorem}
\label{thm:main2}
Let $M$ be a closed symplectic manifold equipped with an atoroidal
symplectic form $\omega$. Assume that a Hamiltonian diffeomorphism
$\varphi_H$ of $M$ has an isolated and homologically non-trivial
one-periodic orbit $\gamma$ with homology class $[\gamma] \neq 0$ in
$\H_1(M;\Z)/\Tor$ and that $\Pp_1(\varphi_H,[\gamma])$ is finite.
Then, for every sufficiently large prime $p$, the diffeomorphism
$\varphi_H$ has a simple periodic orbit with homology class
$p[\gamma]$ and period either $p$ or $p'$, where $p'$ is the first
prime number greater than $p$.
\end{Theorem}
Here, as in Section \ref{sec:intro}, $\Pp_1(\varphi_H,[\gamma])$ is
the collection of one-periodic orbits of $\varphi_H$ with homology
class $[\gamma]$ in $\H_1(M;\Z)/\Tor$.

\begin{Remark}
  In fact, the homology class $p[\gamma]$ in the assertion of the
  theorem can be replaced by the free homotopy class
  $\llbracket\gamma\rrbracket ^ p \defn \llbracket\gamma^p\rrbracket$,
  provided that $\llbracket\gamma\rrbracket \neq 0$ and, for instance,
  the only solutions to the equation $\llbracket\gamma\rrbracket ^ {p}
  = \zeta^{q}$, where $p$ and $q$ are sufficiently large 
  primes and $\zeta$ is a free homotopy class, are
  $\zeta=\llbracket\gamma\rrbracket$ and $p=q$.
\end{Remark}

\begin{proof}[Proof of Theorem \ref{thm:main2}]
  First, recall that all sufficiently large prime numbers $p$ are
  admissible in the sense of \cite{GG:gap}.  Thus, under such
  iterations of $H$, the orbit $\gamma$ stays isolated, and
$$
\HF(\gamma^{p}) = \HF(\gamma)
$$
by the persistence of local Floer homology theorem in
\cite{GG:gap}. In particular, in our case, $\HF(\gamma^{p}) \neq 0$
since $\HF(\gamma)\neq 0$.  From now on, we work with primes that are
large enough and hence admissible, which we order as $p_1< p_2<
\hdots$. In what follows, $p$ or $p_i$ always denotes a prime from
this sequence.

Proceeding with the proof, assume that $p_i$ is a sufficiently large
prime number such that $\varphi_H$ has no simple $p_i$-periodic orbit
in the homology class $p_i[\gamma]$. (If such a prime does not exist
then we are done.)  Our goal is to show that in this case $\varphi_H$
must have a simple $p_{i+1}$-periodic orbit in the homology class
$p_i[\gamma]$.

Let $\gamma_1=\gamma, \gamma_2, \hdots, \gamma_r$ be the set of
one-periodic orbits of $\varphi_H$ such that $[\gamma_l]=[\gamma]$ for
$l=1,\hdots,r$.  Due to the above assumption, any $p_i$-periodic orbit
of $\varphi_H$ with homology class $p_i[\gamma]$ is the $p_i$th
iteration of some one-periodic orbit $\beta$ of $\varphi_H$.  On the
other hand, clearly $[\beta^{p_i}]=p_i[\gamma]$ if and only if
$[\beta]=[\gamma]$. Thus $\beta$ must be one of the orbits
$\gamma_l$. In particular,
$$
\CS \left( H^{\nat p_i}, \llbracket\gamma\rrbracket^{p_i} \right) 
\subset
p_i\bigcup_{l=1}^r  \CS \left( H,\llbracket\gamma_l\rrbracket \right),
$$
where, as in Section \ref{sec:prelim}, $\llbracket\gamma\rrbracket$ is
the free homotopy class of $\gamma$ and
$\llbracket\gamma\rrbracket^{p_i}
\defn
\llbracket\gamma^{p_i}\rrbracket.
$

Assume without loss of generality that $\CA_H(\gamma)=0$, which we can
ensure by adding a constant to $H$; hence $\CA_{H^{\nat
    p}}(\gamma^p)=0$ for all iterations $p$ by
\eqref{eq:action-hom}. Following closely \cite{Gu:hq}, pick $a>0$
outside $ \bigcup_{l=1}^r \, \CS\left(H,\llbracket\gamma_l\rrbracket
\right) $ such that $0$ is the only action value in
$(-a,\,a)$. Therefore, $0$ is the only point of $\CS(H^{\nat p_i},
\llbracket\gamma\rrbracket^{ p_i})$ in the interval $(-p_ia,\,p_ia)$,
and hence
\begin{equation}
\label{eq:locFH}
\HF^{(-p_ia,\,p_ia)}(H^{\nat p_i}, \llbracket\gamma\rrbracket^{ p_i})=
\HF(\gamma^{p_i}) \oplus \hdots,
\end{equation}
where the dots represent the local Floer homology contributions from
the remaining one-periodic orbits $\gamma_l$ with zero action such that
$
\llbracket\gamma_l\rrbracket^{ p_i}
=
\llbracket\gamma\rrbracket^{ p_i}.
$

Let $C = \| H \|$ and choose $p_i$ so large that $p_i a > 6 C
(p_{i+1}-p_i) $.  The latter is guaranteed by the fact that
$p_{i+1}-p_i=o(p_i)$; see, e.g., \cite{BHP}.  Now, pick $\alpha>0$
such that $p_ia- 4C(p_{i+1}-p_i) < \alpha < p_ia- 2C(p_{i+1}-p_i)
$. Then we have
$$
-p_ia< -\alpha < -\alpha + 2 C (p_{i+1}-p_i) < 0 < 
\alpha < \alpha+ 2 C (p_{i+1}-p_i) < p_i a,
$$
and also 
$$
-p_{i+1} a < -\alpha + C (p_{i+1}-p_i) < 0 
< \alpha + C (p_{i+1}-p_i) < p_{i+1} a.
$$

Set $\delta_i:= C (p_{i+1}-p_i)$ and denote by $(a,\,b)+\eps$ the
shifted interval $(a+\eps,\,b+\eps)$.  For a linear homotopy from
$H^{\nat p_i}$ to $H^{\nat p_{i+1}}$, we have the induced map
$$
\HF^{(-\alpha,\,\alpha)}\left( H^{\nat p_i},
  \llbracket\gamma\rrbracket^{p_i} \right) \to
\HF^{(-\alpha,\,\alpha) + \delta_i}\left( H^{\nat p_{i+1}},
  \llbracket\gamma\rrbracket^{p_i} \right)
$$
in filtered Floer homology.  A proof of this fact for the ordinary
filtered Floer homology can be found in \cite{Gi:coiso}. This is also
true in our case since continuation maps preserve homotopy classes of
periodic orbits.  Likewise, the linear homotopy from $ H^{\nat
  p_{i+1}} $ to $H^{\nat p_i}$ results in another action shift in
$\delta_i$.  Consider now the following commutative diagram:
$$
\xymatrix{
\HF^{(-\alpha,\,\alpha)}\left( H^{\nat p_i}, 
\llbracket\gamma\rrbracket^{p_i}\right)
\ar[d] \ar[rrd]^\cong \\ 
\HF^{( -\alpha,\,\alpha) + \delta_i }
\left(H^{\nat p_{i+1}}, 
\llbracket\gamma\rrbracket^{p_i}\right)
\ar[rr]& &
\HF^{(-\alpha,\,\alpha)+ 2 \delta_i}
\left (H^{\nat p_i}, 
\llbracket\gamma\rrbracket^{p_i} \right) 
}
$$
Here the filtered Floer homology groups of $H^{\nat p_i}$ are non-zero
by \eqref{eq:locFH} and the choice of $\alpha$. Furthermore, the
diagonal map is an isomorphism induced by the natural
quotient-inclusion map; see, e.g., \cite{Gi:coiso}. Indeed, zero is
the only action value in the intervals $(-\alpha,\,\alpha)$ and
$(-\alpha,\,\alpha)+2\delta_i$, and the map is an isomorphism by the
stability of the filtered Floer homology; see, e.g., \cite{GG:gap}.

We now have a non-zero isomorphism factoring through the filtered
homology of $H^{\nat p_{i+1}}$ in the commutative diagram. In
particular, this group cannot be zero:
$$ 
\HF^{( -\alpha,\,\alpha) + \delta_i }
\left(H^{\nat p_{i+1}}, 
\llbracket\gamma\rrbracket^{p_i}\right) \neq 0.
$$ 
Thus $\varphi_H$ must have a $p_{i+1}$-periodic orbit $\eta$ in the
homotopy class $\llbracket\gamma\rrbracket^{p_i}$, and hence in the
homology class $p_i[\gamma]$. What remains to show is that $\eta$ is
necessarily simple, provided that $p_{i+1}$ (i.e., $p_i$) is large
enough. Arguing by contradiction, assume that $\eta=\nu^{p_{i+1}}$ for
some one-periodic orbit $\nu$.  Then
\begin{equation}
\label{eq:simple}
p_i[\gamma] = [\eta] = p_{i+1} [\nu].
\end{equation}
Let us write $[\gamma]=n \sigma$ and $[\nu] = m \rho$ for some prime
homology classes $\sigma$ and $\rho$ in $\H_1(M;\Z)/\Tor$, where
$n,\,m \in \N$.  Using \eqref{eq:simple}, we now see that
$\sigma=\rho$ and $n p_i= m p_{i+1}$, where $n$ is fixed -- it is
determined by $\gamma$. This is clearly impossible once $p_{i+1}$ is
greater than $n$.

\end{proof}

\end{document}